\documentclass{article}
\usepackage{amsmath}
\usepackage{amssymb}
\usepackage{xcolor}
\usepackage{natbib}
\usepackage{amsthm}
\usepackage[left=1.2cm, right=1.2cm]{geometry}
\usepackage{url}
\usepackage{hyperref}
\usepackage{graphicx}
\usepackage{placeins}
\usepackage{authblk}

\newtheorem{theorem}{Theorem}
\newtheorem{proposition}{Proposition}
\newtheorem{condition}{Condition}
\newtheorem{corollary}{Corollary}
\newtheorem{lemma}{Lemma}
\newcommand{\convP}{\overset{P}{\to}}

\author[1]{Ezequiel Smucler}
\author[1]{Ludovico Lanni}
\author[1]{David Masip}
\affil[1]{Glovo}

\begin{document}

\title{A note on the properties of the confidence set for the local average treatment effect obtained by inverting the score test}
\date{\today}
\maketitle

\begin{abstract}
    We study the properties of the score confidence set for the local average treatment effect in non and semiparametric instrumental variable models. This confidence set is constructed by inverting a score test based on an estimate of the nonparametric influence function for the estimand, and is known to be uniformly valid in models that allow for arbitrarily weak instruments; because of this, the confidence set can have infinite diameter at some laws.
    We characterize the six possible forms the score confidence set can take: a finite interval, an infinite interval (or a union of them), the whole real line, an empty set, or a single point. 
    Moreover, we show that, at any fixed law, the score confidence set asymptotically coincides—up to a term of order $1/n$—with the Wald confidence interval based on the doubly robust estimator which solves the estimating equation associated with the nonparametric influence function. This result implies that, in models where the efficient influence function coincides with the nonparametric influence function, the score confidence set is, in a sense, optimal in terms of its diameter.
    We also show that under weak instrument asymptotics, where the strength of the instrument is modelled as local to zero, the doubly robust estimator is asymptotically biased and does not follow a normal distribution. A simulation study confirms that, as expected, the doubly robust estimator performs poorly when instruments are weak, whereas the score confidence set retains good finite-sample properties in both strong and weak instrument settings.
    Finally, we provide an algorithm to compute the score confidence set, which is now available in the \texttt{DoubleML} package for double machine learning.
\end{abstract}

\section{Introduction}

Weak instruments, that is, instruments that are weakly correlated with treatment, are common in practice.
For example, \cite{lee2022valid} report that the fraction of instrumental variable models published in the American Economic Review between 2013 and 2019 that suffer from weak instruments is substantial.
On the other hand, in online experimentation (`A/B tests'), weak instruments arise routinely in experiments in which the instrument, which is randomized at the user level, is some form of encouragement to use a feature, with feature usage serving as the treatment. The causal effect of interest is that of feature usage on a given outcome. In such settings, the instrument may be only weakly correlated with treatment, since the encouragement is not always effective. See \cite{spotify2023} and \cite{Twitch2017} for example.

In the presence of weak instruments, Wald confidence intervals for the local average treatment effect are not uniformly valid \citep{dufour, gleser, smucler2025asymptotic}. Thus, for any given sample size, there are laws under which the Wald confidence interval has coverage that can be far below the nominal below; for this reason their use is typically discouraged  whenever the dependence between the instrument and the treatment is suspected to be weak \citep{andrews2019weak}. Moreover, popular methods to construct point estimators for the local average treatment effect, such as the two stage least squares estimator, are known to be biased under weak instruments. More precisely,  \cite{staiger1997instrumental} derived the asymptotic distribution of the two stage least squares estimator under weak instrument asymptotics, that is, assuming that the strength of the instrument is of order $1/\sqrt{n}$, where $n$ is the sample size, and assuming a linear model. The limiting distribution is non-normal and biased. These issues have prompted the development of methods for the construction of confidence sets which are robust to weak instruments.

For linear instrumental variable models, a common approach to obtain uniformly valid confidence sets is to invert the Anderson-Rubin test \citep{anderson1949,andrews2006optimal,mikusheva2010,andrews2019weak}.
Recently, \cite{ma2023}, building on ideas from \cite{stock2000gmm} and \cite{andrews2016conditional}, proposed an extension of the Anderson-Rubin test suitable for non or semiparametric instrumental variable models. This extension, in the spirit of doubly robust/double machine learning (DRML) methods \citep{van2006targeted, chernozhukov2018double} is based on using cross-fitting and flexible machine learning algorithms to estimate the nuisances that appear in the non-parametric influence function of the functional that identifies the local average treatment effect, and then using this estimated influence function as a score. The resulting score test is then inverted to construct a confidence set, which \cite{ma2023} has shown to be uniformly valid over models that allow for arbirtrarily weak instruments. In what follows, we will refer to this confidence set as the score confidence set.

Our paper contributes to the literature on weak instruments in several ways. First, we show that the score confidence set can take one of six forms: the empty set, a finite interval, an infinite interval, a union of two infinite intervals, the entire real line, or a single point. Interestingly, we show that the score confidence set of asymptotic level $1-\alpha$ has infinite diameter if and only if the score test for the hypothesis that the average effect of the instrument on the treatment is zero does not reject the null at the $\alpha$ level. Second, we show that, at any fixed law, the score confidence set coincides, up to a term of order $1/n$ and with probability tending to one, with the Wald confidence interval based on the estimator proposed by \cite{chernozhukov2018double, takatsu}. The latter is the DRML estimator obtained by solving the estimating equation associated with the nonparametric influence function. Our result implies that, given a model $\mathcal{M}$ and a law $P\in\mathcal{M}$ such that the efficient influence function at $P$ is equal to the nonparametric influence function, with probability tending to one, the diameter of the score confidence set is smaller than or equal to the diameter of the Wald confidence interval constructed from any estimator that is regular and asymptotically linear over $\mathcal{M}$. Thus, in a sense, the score confidence set is optimal. 
Third, we extend the result of \cite{staiger1997instrumental} to nonparametric instrumental variable models by deriving the asymptotic distribution of the DRML estimator under weak instrument asymptotics. The latter distribution is non-normal and has infinite mean.
Fourth, we report results from a simulation study that demonstrate the good finite-sample performance of the score confidence set. We also apply it to data from a real-world experiments conducted at Glovo, the leading on-demand delivery platform in Southern Europe and Northern Africa.
Finally, we provide an implementation of an algorithm to compute the score confidence set, available in the \href{https://docs.doubleml.org/stable/index.html}{DoubleML} \citep{bach2022doubleml} package for double machine learning. 

Given our theoretical and empirical results, the uniformity result of \cite{ma2023}, and the readily available software to compute it, we believe that the score confidence set is a valuable addition to the practitioners toolbox.

\section{Setup}\label{sec:setup}

Let $O=(Y,A,Z,X)$, where $Y$ is a scalar outcome, $A$ a binary treatment, $Z$ a binary instrument, and $X$ a vector of covariates. Let $P$ be the law of $O$. We are interested in inference for the functional
$$
\varphi(P):= \frac{E_{P}\left\lbrace E_{P}(Y\mid Z=1,X) - E_{P}(Y\mid Z=0,X)\right\rbrace}{E_{P}\left\lbrace E_{P}(A\mid Z=1,X) - E_{P}(A\mid Z=0,X)\right\rbrace},
$$
under models that allow for weak instruments, that is, models that include laws $P$ under which 
$$
E_{P}\left\lbrace E_{P}(A\mid Z=1,X)-E_{P}(A\mid Z=0,X)\right\rbrace$$
 is close to zero.

The functional $\varphi(P)$ coincides with the so-called local average treatment effect under certain causal assumptions, which for completeness we will now review.
Let  $Y(a)$ denote the potential outcome that would have been observed had $A$ been set to $a$. 
We write $A(z)$, $Y(z)$ for the potential outcomes when setting $Z=z$, for $A$ and
$Y$ respectively. Finally, we let $Y(z, a)$ denote the potential outcome of $Y$ when $Z=z$ and $A=a$.
The local average treatment effect is defined as $E\lbrace Y(1)-Y(0)\mid A(1) > A(0)\rbrace$. This is the average treatment effect among the so-called compliers, that is, the units that satisfy $A(1) > A(0)$.
Under the assumptions that $Y(z,a)=Y(z)$ for all $z,a\in\lbrace 0,1\rbrace^{2}$, $(Y(z),A(z))$ is independent of $Z$ given $X$, $P\left\lbrace A(1)\geq A(0)\right\rbrace =1$, $P\lbrace A(1)=A(0)\rbrace<1$ and $P(Z=1\mid X)\in (0,1)$ almost surely, the 
local average treatment effect is equal to $\varphi(P)$ \citep{imbens}.

The influence function of $\varphi(P)$ plays a central role in defining both the DRML estimator and the score confidence set. We now revisit this influence function in the context of the nonparametric instrumental variable model. For a formal treatment of influence functions and their role in causal inference, see the review by \cite{kennedy2024semiparametric}.
Define the nuisance functions
\begin{align*}
    &g_{P}(Z,X)=E_{P}(Y\mid Z,X), \quad
    m_{P}(Z\mid X)=P(Z\mid X), \quad
    r_{P}(Z,X)=E_{P}(A\mid Z, X),
\end{align*}
and let $\eta_{P}=(m_{P},r_{P},g_{P})$.
In a nonparametric model, the influence function for $\varphi$
at $P$ is given by \citep{ogburn,wang2018bounded}
\begin{equation}
    \varphi^{1}_{P}(O)=\frac{\psi_{b,\eta_{P}}-\varphi(P)\psi_{a,\eta_{P}}}{E_{P}\left\lbrace r_{P}(1,X) - r_{P}(0,X)\right\rbrace} ,
    \label{eq:eff_IF}
\end{equation}
 where for any $\eta=(m,r,g)$
 $$
\psi_{b,\eta}=\frac{2Z-1}{m(Z\mid X)} \left\lbrace Y-g(Z,X)\right\rbrace +g(1,X) - g(0,X), \quad \psi_{a,\eta}=\frac{2Z-1}{m(Z\mid X)} \left\lbrace A-r(Z,X)\right\rbrace +r(1,X) - r(0,X) .
$$
Note that 
$$
\varphi(P)=\frac{E_{P}\left\lbrace \psi_{b,\eta_{P}}\right\rbrace}{E_{P}\left\lbrace \psi_{a,\eta_{P}}\right\rbrace}.
$$

\section{The score confidence set}\label{sec:score_confidence_set}

Suppose we have $n$ independent identically distributed copies of $O$ and let $\widehat{\eta}$ be an estimator of $\eta_{P}$. Let $\mathbb{P}_{n}$ denote the empirical mean operator. The score confidence set is given by
$C_{n}:=\left\lbrace \theta : \vert S_{n}(\theta)\vert \leq z_{1-\alpha/2}\right\rbrace$, where
$$
S_{n}(\theta) = \frac{\sqrt{n} \mathbb{P}_{n} \left[   \psi_{b,\widehat{\eta}}  -\theta\psi_{a,\widehat{\eta}}\right]}{\mathbb{P}^{1/2}_{n} \left\lbrace \psi_{b,\widehat{\eta}}  -\theta\psi_{a,\widehat{\eta}} \right\rbrace^{2} },
$$
and $z_{1-\alpha}$ is the $1-\alpha$ quantile of the standard normal distribution.
Thus, $C_{n}$ is obtained by inverting the test for $\varphi(P)=\theta$ with test statistic $S_{n}(\theta)$ and critical value $z_{1-\alpha/2}$. When $S_{n}$ is computed using cross-fitting, and $\widehat{\eta}$ is obtained leveraging flexible machine learning estimators, $C_{n}$ coincides with the confidence set studied in \cite{ma2023}. See, in particular, footnote 8 in page 16 of their paper. 
\cite{ma2023} shows that $C_{n}$ is uniformly valid over models that allow for arbitrarily weak instruments. More precisely, they show that for certain models $\mathcal{M}$ that allow for arbitrarily weak instruments it holds that 
$$  
\liminf_{n}\inf_{P\in\mathcal{M}} P\left\lbrace \varphi(P)\in C_{n}\right\rbrace \geq 1-\alpha.
$$
For practitioners, using confidence sets that are uniformly valid over rich models is important, because, for non-uniform confidence sets, at any finite sample size there exist laws under which the coverage of the true parameter is less than the nominal level.

Next, we will show that $C_{n}$ can take six possible forms. 
Note that $\vert S_{n}(\theta) \vert \leq z_{1-\alpha/2}$ if and only if $a\theta^{2}+b\theta+c\leq 0$,
where
\begin{align*}
a := n\left[ \mathbb{P}_{n} \psi_{a,\widehat{\eta}}\right]^{2} - z_{1-\alpha/2}^{2} \mathbb{P}_{n} \left[ \psi_{a,\widehat{\eta}}\right]^{2},
\quad
b : = -2n\mathbb{P}_{n} \psi_{a,\widehat{\eta}}\mathbb{P}_{n} \psi_{b,\widehat{\eta}} + 2z_{1-\alpha/2}^{2} \mathbb{P}_{n} \left[ \psi_{a,\widehat{\eta}}\psi_{b,\widehat{\eta}}\right], \quad
c := n\left[ \mathbb{P}_{n} \psi_{b,\widehat{\eta}}\right]^{2} - z_{1-\alpha/2}^{2} \mathbb{P}_{n} \left[ \psi_{b,\widehat{\eta}}\right]^{2}.
\end{align*}
Thus, $C_{n}$ is the set of $\theta$ that satisfy the quadratic inequality $a\theta^{2}+b\theta+c\leq 0$.
Let $\Delta := b^2 - 4ac$. In the case in which $\Delta\geq 0$, we let $r_{1}:=(-b- \sqrt{\Delta})/(2a)$, $r_{2}:=(-b+ \sqrt{\Delta})/(2a)$. We then have the following proposition.
\begin{proposition}\label{prop:five_forms}
    $C_{n}$ can take one of the following forms:
    \begin{enumerate}
    \item If $\Delta > 0$ and $a>0$, then $C_{n}$ is the interval $[r_{1},r_{2}]$.
    \item If $\Delta > 0$ and $a<0$, then $C_{n}$ is the union of two intervals $(-\infty,r_{2}]\cup [r_{1},\infty)$.
    \item If $\Delta < 0$ and $a>0$, then $C_{n}$ is the empty set.
    \item If $\Delta < 0$ and $a<0$, then $C_{n}$ is the whole real line.
    \item If $\Delta \neq 0$ and $a=0$, then $C_{n}=(-\infty,-c/b]$ if $b>0$ and $C_{n}=[-c/b,\infty)$ if $b<0$.
    \item If $\Delta = 0$ and $a\neq 0$, then $C_{n}$ is a single point $\left\lbrace -b/(2a)\right\rbrace$.
    \item If $\Delta = 0$ and $a= 0$ then $C_{n}$ is the whole real line if $c\leq 0$ and the empty set if $c>0$.
    \end{enumerate}
\end{proposition}
\cite{mikusheva2010} proved a similar result for the confidence set obtained by inverting the Anderson-Rubin test. We highlight that for a confidence set to be uniformly asymptotically valid over models that allow for arbitrarily weak instruments, it is necessary for the set to have infinite length with high probability, under some laws \citep{gleser, dufour, smucler2025asymptotic}. Thus, the fact that the score confidence set can be equal to an infinite interval, a union of two infinite intervals or to the whole real line, is not a problem but rather a desirable property which, as stated in the introduction, Wald confidence intervals do not enjoy. In cases in which the score confidence set has infinite diameter, the analyst should interpret the result as indicating that the data is not informative about the value of the local average treatment effect, possibly due to the presence of weak instruments. In fact, note that, excluding the trivial case in which $a=b=0$ and $c>0$, the score confidence set has infinite diameter if and only if $a \leq 0$, equivalently, if and only if $D_{n}(0) \leq z_{1-\alpha/2}^{2}$, where
$$
D_{n}(\theta) := \frac{n\left[\mathbb{P}_{n} \psi_{a,\widehat{\eta}} - \theta\right]^{2}}{\mathbb{P}_{n} \left[\psi_{a,\widehat{\eta}} - \theta\right]^{2}}.
$$
The random variable $D_{n}(0)$ is a test statistic for the null hypothesis $E_{P}\left\lbrace \psi_{a,\eta_{P}} \right\rbrace = 0$; this test will have asymptotic level $\alpha$ under very mild assumptions, see for example Theorem 5.1 of \cite{chernozhukov2018double}.  When the null hypothesis $E_{P}\left\lbrace \psi_{a,\eta_{P}} \right\rbrace = 0$ holds, the instrument is invalid because its average effect on the treatment is zero. Thus, the score confidence set has infinite diameter precisely when the null hypothesis that the instrument is invalid is not rejected at the $\alpha$ level.

\section{Asymptotic behaviour of the score confidence set at a fixed law}\label{sec:asymp}

Next, we will study the asymptotic behaviour of the score confidence set at a fixed law. In what follows, we will say that an estimator $\widehat{\phi}$ of $\varphi(P)$ is regular with respect to a model $\mathcal{M}$ at a law $P\in\mathcal{M}$ if it converges to its limiting distribution locally uniformly
over laws contiguous to $P$ \citep{van2000asymptotic}. We will say the estimator is asymptotically linear if it can be expressed as $\sqrt{n}\left\lbrace \widehat{\phi}-\varphi(P)\right\rbrace=n^{-1/2}\sum\limits_{i=1}^{n} \phi_{P}^{1}(O)+o_{P}(1)$, where $\phi_{P}^{1}(O)$ has zero mean and finite variance under $P$. If $\widehat{\phi}$ is asymptotically linear, then it is asymptotically normal, with asymptotic variance given by $var_{P}\left\lbrace \phi_{P}^{1}(O)\right\rbrace$.

Define
$$
\widehat{\varphi}:=\frac{\mathbb{P}_{n}\psi_{b,\widehat{\eta}}}{\mathbb{P}_{n}\psi_{a,\widehat{\eta}}}.
$$
When $\widehat{\varphi}$ is computed using cross-fitting, and
$\widehat{\eta}$ is built using flexible non-parametric estimators, the estimator $\widehat{\varphi}$ coincides with the so-called DRML estimator proposed in \cite{chernozhukov2018double, takatsu}.  The authors of these papers showed that, if
\begin{align}
    &\widehat{\eta} \text{ is computed on an independent sample},
    \label{eq:sample_split}
    \\
    & Y \text{ is almost surely bounded under } P,
    \label{eq:Y_bounded}
    \\
    & \widehat{m} \text{ and }  m_{P} \text{ are bounded away from zero almost surely under } P,
    \label{eq:posit}
    \\
    &\Vert \widehat{m} - m_{P}\Vert_{L^{2}(P)} \max\left\lbrace \Vert \widehat{g} - g_{P}\Vert_{L^{2}(P)}, \Vert \widehat{r} - r_{P}\Vert_{L^{2}(P)}  \right\rbrace = o_{P}(n^{-1/2}),
    \label{eq:rates1}
    \\
    & \Vert \widehat{m} - m_{P}\Vert_{L^{2}(P)}=o_{P}(1),\Vert \widehat{g} - g_{P}\Vert_{L^{2}(P)}=o_{P}(1), \Vert \widehat{r} - r_{P}\Vert_{L^{2}(P)}=o_{P}(1),
    \label{eq:rates2}
\end{align}
and $E_{P}\psi_{a,\eta_{P}}\neq 0$, then $\widehat{\varphi}$ is regular and asymptotically linear, and
$$
\widehat{\sigma}^{2}:=\frac{\mathbb{P}_{n}\left( \psi_{b,\widehat{\eta}} - \widehat{\varphi}\psi_{a,\widehat{\eta}}\right)^{2}}{\mathbb{P}^{2}_{n} \left( \psi_{a,\widehat{\eta}}\right)},
$$
is a consistent estimator of its asymptotic variance, which is equal to the variance of the nonparametric influence function $\varphi^{1}_{P}(O)$. 
In the equations above,  for any function $h$ of $O$, $\Vert h\Vert^{2}_{L^{2}(P)}$ is defined as $\int h^{2}(o) dP$ and for any sequence of random variables $(V_{n})_{n\geq 1}$ we write $V_{n}=o_{P}(1)$ if for all $\epsilon>0$, $P(\left\vert V_{n}\right\vert > \epsilon)\to 0$ as $n\to\infty$. We highlight that the requirement that $Y$ be bounded almost surely can be relaxed to a moment condition.
The requirement in equation \eqref{eq:rates1} is a  rate double-robustness condition \citep{unified,mixedbias}. 
It requires that either $m_{P}$, or both $g_{P}$ and $r_{P}$, be estimated  at a sufficiently fast rate. In particular, it allows, for example, for the rate of estimation of $g_{P}$ and $r_{P}$ to be arbitrarily slow, as long as $m_{P}$ is estimated at a rate of order $n^{-1/2}$. 
When $m_{P}$ is known and bounded away from zero, as is the case in a randomized experiment, setting $\widehat{m} = m_{P}$ ensures that \eqref{eq:sample_split}--\eqref{eq:rates2} require only that $Y$ is bounded and that $\widehat{g}$ and $\widehat{r}$ are consistent estimators of $g_{P}$ and $r_{P}$, respectively.

Let $
W_{n}:=[\widehat{\varphi}-z_{1-\alpha/2}(\widehat{\sigma}/\sqrt{n}),\widehat{\varphi}+z_{1-\alpha/2}(\widehat{\sigma}/\sqrt{n})]
$ be the Wald confidence interval for $\varphi(P)$ based on $\widehat{\varphi}$. 
It follows that, under the assumptions \eqref{eq:sample_split}-\eqref{eq:rates2}, $W_{n}$
is pointwise asymptotically valid, that is, $P\left\lbrace \varphi(P)\in W_{n}\right\rbrace \to 1-\alpha$ as $n\to\infty$.

In the following theorem we show that, at any  law $P$ where \eqref{eq:sample_split}-\eqref{eq:rates2} hold, with probability tending to one the score confidence set $C_{n}$ is equal to the Wald confidence interval $W_{n}$, except for a term of order $1/n$.

\begin{theorem}\label{theo:main}
    Assume that \eqref{eq:sample_split}-\eqref{eq:rates2} hold under $P$ and that
    $
    E_{P}\psi_{a,\eta_{P}}\neq 0.
    $
    Then, when $n\to\infty$
    $$
        P\left\lbrace C_{n} = \left[\widehat{\varphi}-z_{1-\alpha/2}(\widehat{\sigma}/\sqrt{n})+O_{P}(1/n),\widehat{\varphi}+z_{1-\alpha/2}(\widehat{\sigma}/\sqrt{n})+O_{P}(1/n)\right]\right\rbrace \to 1.
    $$
\end{theorem}

Consider now a model $\mathcal{M}$ such that the efficient influence function for $\varphi$ at any $P\in\mathcal{M}$ is given by the nonparametric influence function $\varphi^{1}_{P}$. An example of such a model is the one consisting of laws $P$ such that $m_{P}(Z\mid X)=\pi(Z\mid X)$, for some known function $\pi$; this follows immediately from the fact that $m_{P}(Z\mid X)$ is ancilliary for $\varphi(P)$. This is the model that arises when the instrument is randomized, possibly conditionally on $X$.  
The following corollary essentially states that the diameter of the score confidence set is asymptotically smaller than that of the Wald confidence interval constructed from any estimator that is regular and asymptotically linear under $\mathcal{M}$, but distinct from the DRML estimator $\widehat{\varphi}$.

\begin{corollary}\label{coro:main}
    Let $P\in\mathcal{M}$ be such that the assumptions in Theorem \ref{theo:main} hold.
    Let $\widetilde{\varphi}$ be a regular and asymptotically linear estimator of $\varphi$ at $P$, and let   $\widetilde{s}^{2}$ be a consistent estimator of $s^{2}_{P}$, defined as the asymptotic variance of $\widetilde{\varphi}$. Assume that $\sigma^{2}_{P}\neq s^{2}_{P}$.
    Let $I_{n}=\left[\widetilde{\varphi}-z_{1-\alpha/2}(\widetilde{s}/\sqrt{n}),\widetilde{\varphi}+z_{1-\alpha/2}(\widetilde{s}/\sqrt{n})\right]$ be the Wald confidence interval constructed from $\widetilde{\varphi}$ and $\widetilde{s}$. Then 
    $$
        P\left\lbrace diam(C_{n}) < diam(I_{n}) \right\rbrace \to 1 \text{ as } n\to\infty.
        $$
\end{corollary}

\section{Weak instrument asymptotics for the DRML estimator}\label{sec:weak_iv_asympotics}

In this section, we will study the asymptotic behaviour of the DRML estimator of $\varphi(P)$ under weak instrument asymptotics. In particular, we will assume that for each $n$, we observe $O_{1},\dots,O_{n}$ with law $P_{n}$. The assumptions we impose on $P_{n}$ are summarised in the following condition. For any law $P$ we let $\Sigma_{P,ab}$ be the covariance matrix of the random vector $(\psi_{a,\eta_{P}},\psi_{b,\eta_{P}})$.

\begin{condition}\label{cond:weak_IV}
There exist fixed constants $c_{a},c_{b}, c_{\lambda}$ and $c$ such that $c_{a}$ and $c_{b}$ are non-zero, $c_{\lambda}$ and $c$ are positive and
the sequence of laws $(P_{n})_{n\geq 1}$ satisfies that 
\begin{enumerate}
    \item $E_{P_{n}}\left\lbrace \psi_{a,\eta_{P_{n}}}\right\rbrace=c_{a}/\sqrt{n}$ and $E_{P_{n}}\left\lbrace \psi_{b,\eta_{P_{n}}}\right\rbrace=c_{b}/\sqrt{n}$ for all $n$,
    \item $\Sigma_{P_{n},ab}$ converges to an invertible matrix $\Sigma_{ab}$,
    \item $\left\vert \psi_{a,\eta_{P_{n}}}\right\vert\leq c$ and $\left\vert \psi_{b,\eta_{P_{n}}}\right\vert\leq c$ almost surely under $P_{n}$ for all $n$.
\end{enumerate}
\end{condition}
Note that by part 1 of Condition \ref{cond:weak_IV}, $\varphi(P_{n})=c_{b}/c_{a}$ for all $n$. In words, Condition \ref{cond:weak_IV} states that, under $P_{n}$, the strenght of the instrument, as measured by 
$
E_{P_{n}}\left\lbrace \psi_{a,\eta_{P_{n}}}\right\rbrace,
$
 is of order $1/\sqrt{n}$, that the estimand of interest does not vary with $n$, that $\Sigma_{P_{n},ab}$ converges to an invertible matrix $\Sigma_{ab}$, and that  $\psi_{a,\eta_{P_{n}}}$ and $\psi_{b,\eta_{P_{n}}}$ are uniformly bounded by a constant. This last requirement will hold, for example, when $Y$ is bounded and there exists a constant $c_{m}$ such that $m_{P_{n}}(Z,X)\geq c_{m}$ with probability one under $P_{n}$.

We will also require the following condition. For any sequence of random variables $(V_{n})_{n\geq 1}$ we write $V_{n}=o_{P_{n}}(1)$ if for any $\epsilon>0$, $P_{n}(\left\vert V_{n}\right\vert > \epsilon)\to 0$ as $n\to\infty$.
\begin{condition}\label{cond:nuisances}
    For any sequence of laws $(P_{n})_{n\geq 1}$ that satisfies Condition \ref{cond:weak_IV}, it holds that 
    \begin{align*}
        \sqrt{n} \mathbb{P}_{n}\left( \psi_{a,\widehat{\eta}} - \psi_{a,\eta_{P_{n}}} \right) = o_{P_{n}}(1),\quad \sqrt{n} \mathbb{P}_{n}\left( \psi_{b,\widehat{\eta}} - \psi_{b,\eta_{P_{n}}} \right) = o_{P_{n}}(1).
    \end{align*}
\end{condition}
Note that $\sqrt{n} \mathbb{P}_{n}\left( \psi_{a,\widehat{\eta}} - \psi_{a,\eta_{P_{n}}} \right)$ can be expressed as
$$
\sqrt{n} \left\lbrace \mathbb{P}_{n}\left( \psi_{a,\widehat{\eta}} - \psi_{a,\eta_{P_{n}}} \right) - E_{P_{n}}\left( \psi_{a,\widehat{\eta}} - \psi_{a,\eta_{P_{n}}} \right) \right\rbrace+ \sqrt{n}E_{P_{n}}\left( \psi_{a,\widehat{\eta}} - \psi_{a,\eta_{P_{n}}}\right),
$$
and similarly for  $\sqrt{n} \mathbb{P}_{n}\left( \psi_{b,\widehat{\eta}} - \psi_{b,\eta_{P_{n}}} \right)$. 
The first term in the last display is usually called the asymptotic equicontinuity term, and the second term  the bias term \citep{kennedy2024semiparametric}. Thus, Condition \ref{cond:nuisances} requires that the asymptotic equicontinuity term and the bias term in the expansions of $\mathbb{P}_{n}\psi_{a,\widehat{\eta}}$ and $\mathbb{P}_{n}\psi_{b,\widehat{\eta}}$ are both $o_{P_{n}}(1)$. This will hold when $\widehat{\eta}$ is computed on an independent sample and \eqref{eq:Y_bounded}, \eqref{eq:posit}, \eqref{eq:rates1} and \eqref{eq:rates2} hold under $P_{n}$ instead of $P$.
See the proof of Lemma 3.2 of \cite{takatsu} for details.

In the following theorem we show that when Conditions \ref{cond:weak_IV} and \ref{cond:nuisances} hold, under weak instrument asymptotics the DRML estimator $\widehat{\varphi}$ converges in distribution to a non-normal random variable. The limiting distribution depends on the limiting covariance matrix $\Sigma_{ab}$ and the constants $c_{a}$ and $c_{b}$ that appear in Condition \ref{cond:weak_IV}.

\begin{theorem}\label{theo:weak_IV}
Let $(P_{n})_{n\geq 1}$ satisfy Condition \ref{cond:weak_IV} and assume Condition \ref{cond:nuisances} holds. Then the distribution of 
$\widehat{\varphi} - \varphi(P_{n})$ converges to the distribution of
\begin{equation}
\frac{c_{a}N_{b}+c_{b}N_{a}}{c_{a}^{2}-c_{a}N_{a}},
\label{eq:asymp_dist}
\end{equation}
where $(N_{a},N_{b})$ has distribution $N(0, \Sigma_{ab})$.
\end{theorem}
In particular, Theorem \ref{theo:weak_IV} implies that, under weak instrument asymptotics, the asymptotic bias of $\widehat{\varphi}$ is infinite, since the distribution of the random variable in \eqref{eq:asymp_dist} does not have a finite mean \citep{marsaglia}.

\section{Simulations}\label{sec:sim}

In this section, we will report results from a simulation study that demonstrates the good finite-sample performance of the score confidence set. We consider the following data-generating process:
\begin{align*}
    &U \sim N(0,1),
    \\
    &X \sim N(0,1),
    \\
    &Z \sim \text{Bernoulli}(0.5),
    \\
    &A = I\lbrace \pi \times Z \times I\lbrace X >0\rbrace + U \rbrace,
    \\
    & Y =  2\times \text{sign}(U).
\end{align*}
Note that the local average treatment effect is zero. The constant $\pi$ parametrizes the strength of the instrument. We consider a weak instrument setting, where we take $\pi=0.15/\sqrt{n}$ and a strong instrument setting, where we take $\pi=5$. We will take $n\in\lbrace 1500,4500,7500,10500,12000\rbrace$ and run 1000 replications for each value and $n$ and $\pi$. 

We compare the score confidence set with the Wald confidence interval constructed using the DRML estimator and use the implementation of both methods available in the \href{https://docs.doubleml.org/stable/index.html}{DoubleML} \citep{bach2022doubleml} package for double machine learning; the implementation of the score confidence set was recently contributed to the package by the authors of this paper. To fit the nuisance functions, we will use: for $g_{P}$ a linear regression, for $r_{P}$ a random forest classifier and for $m_{P}$ the true $P(Z\mid X)$ function, which in this case equals 0.5 for all values of $X$ and $Z$. We use the linear regression and random forest implementations available in the \href{https://scikit-learn.org/stable/}{Scikit-learn} \citep{sklearn} package, with default hyperparameters. Code to replicate the simulations is available at \url{https://github.com/david26694/simulations-score-confidence-set}.

We report the average coverage, as well as the median length for each of the confidence sets. We see in Figure \ref{fig:weak_instrument_coverage} that the average coverage of the score confidence set is very close to the nominal 0.95 level in the weak instrument setting, for all values of $n$, whereas the coverage of the DRML Wald confidence interval is much lower, in line with Theorem \ref{theo:weak_IV}. The median length of the score confidence set is infinity, for all values of $n$; this is expected since in this scenario there is essentially no information about the estimand of interest in the data. For this reason we don't plot the median length of the confidence sets in this case. On the other hand, in the strong instrument setting both methods behave similarly, as predicted by Theorem \ref{theo:main}. In Figure \ref{fig:strong_instrument_coverage}, we see that both methods have empirical coverage close to the nominal level. Figure \ref{fig:strong_instrument_length} shows that the median length of the score confidence set is very similar to that of the DRML Wald confidence interval, and the difference between the two decreases as the sample size increases, as predicted by Theorem \ref{theo:main}.

\begin{figure}[h!]
    \centering
    \includegraphics[scale=0.7]{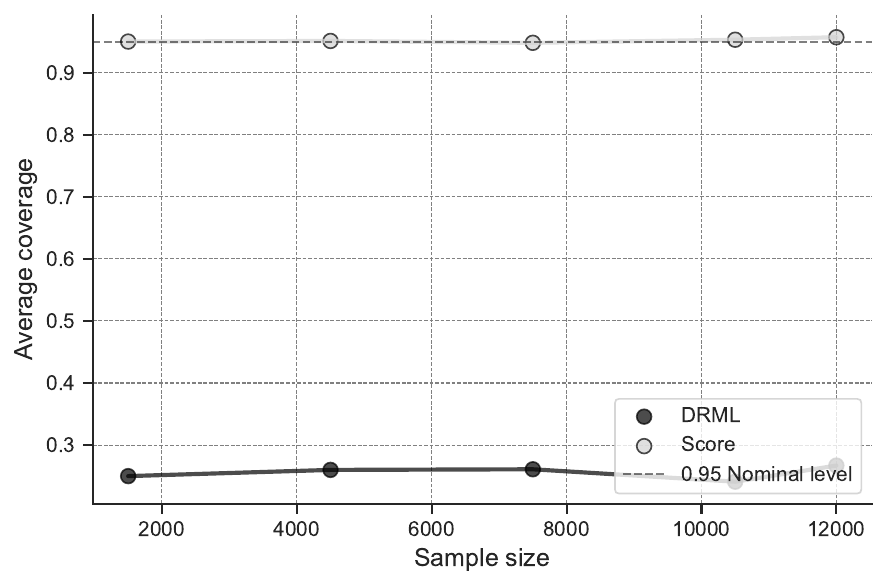}
    \caption{Empirical coverage of the score confidence set and the DRML Wald confidence interval in the weak instrument setting.}
    \label{fig:weak_instrument_coverage}
\end{figure}


\begin{figure}[h!]
    \centering
    \includegraphics[scale=0.7]{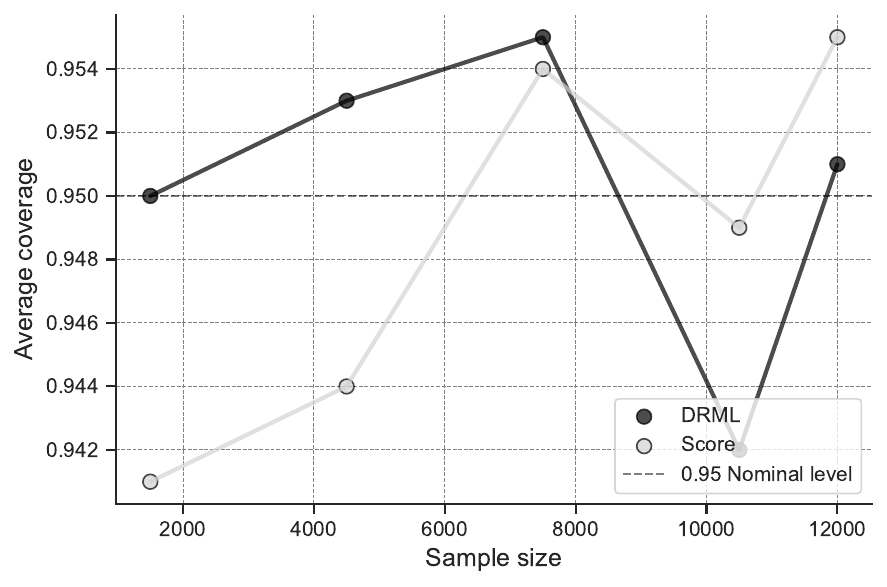}
    \caption{Empirical coverage of the score confidence set and the DRML Wald confidence interval in the strong instrument setting.}
    \label{fig:strong_instrument_coverage}
\end{figure}

\begin{figure}
    \centering
    \includegraphics[scale=0.7]{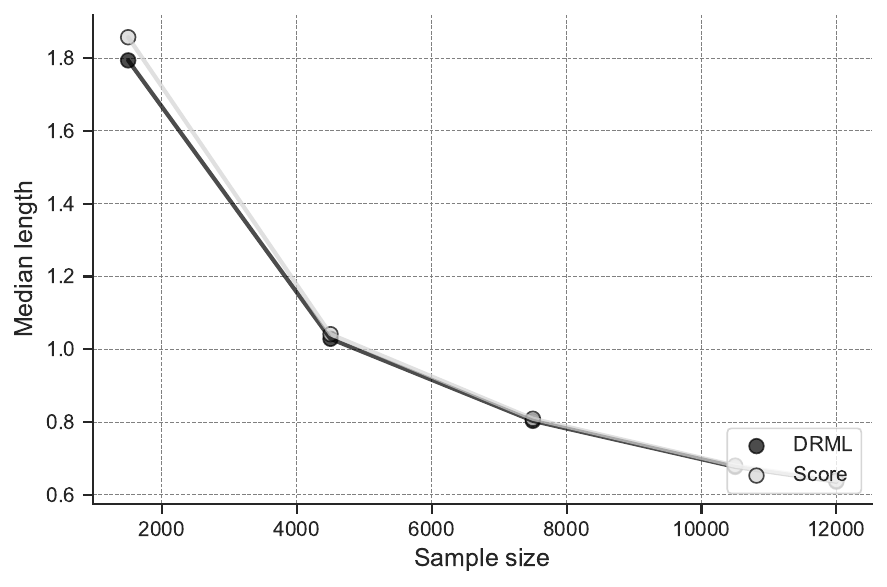}
    \caption{Median length of the score confidence set and the DRML Wald confidence interval in the strong instrument setting.}
    \label{fig:strong_instrument_length}
\end{figure}

\FloatBarrier

\section{An analysis of Glovo experiments}\label{sec:glovo}
In this section, we will analyse a set of  experiments conducted at Glovo using both the score confidence set and the Wald confidence intervals based on the DRML estimator. Glovo is an on-demand delivery platform headquartered in Barcelona. It operates as a three-sided marketplace, connecting three distinct user groups: customers seeking the convenience of local delivery, a wide array of partner establishments (including restaurants, grocery stores, pharmacies, and retail shops) aiming to expand their reach and sales, and couriers who use the platform to earn income by providing delivery services. 

Within the Glovo ecosystem, customers can take certain actions that influence their long-term behavior and enhance platform engagement. These actions are often taken by customers after being offered targeted promotional prices, aimed at encouraging particular user activities. 
Quantifying the monetary impact associated with a customer performing a specific action in the app is crucial for the company. This allows the business to estimate the return on investment for various incentive programs, thereby enabling data-driven decisions on how much to invest in different levers. To achieve this, Glovo employs an experimental encouragement design. In this framework, within a defined geographical segment, a randomly selected portion of users eligible for a particular incentive (e.g., a discount) is intentionally excluded from receiving it. The remaining eligible users receive the incentive. This setup naturally lends itself to an instrumental variables estimation approach. Here, the treatment is an indicator of whether the action being analysed was performed (e.g., placing an incentivized order). The instrument is the assignment to receive or not receive the incentive, a factor that influences the treatment but is assumed not to directly affect the outcome. The outcome variable can be taken, for instance, as the total number of orders placed by the customer over the experiment's duration. These experiments typically run for periods ranging from one to six months, providing a window to observe subsequent customer behavior. In such a setup, the quantity one wants to make inference on is the local average treatment effect.

In the experiments we just described, additional information on the participating customers is available in the form of pre-experimental covariates $X$. For instance, $X$ can be taken as the number of orders placed by the customer in the 30 days before the experiment started. These covariates can be leveraged to obtain more efficient estimators of the estimand of interest. We will consider experiments on two special incentives for the placement of an order; we refer to the incentives as Promo 1 and Promo 2. The experiments ran during a full semester in 2024. For Promo 1 (Promo 2) we take a single covariate $X$, equal to the number of  orders placed by the customer in the 30 days (60 days) before the experiment started.

We used the same specification for the nuisance function as in the simulations, and again used the implementation of the methods available in the DoubleML package.
Our analysis for the Promo 1 experiments shows that in 1 out of the 22 experiments we considered, the score confidence set was unbounded and equal to the whole real line. This case is also particularly interesting because the DRML confidence interval not only is bounded, but also excludes zero. The fact that the score confidence set is unbounded indicates that the data is not informative about the value of the local average treatment effect, possibly due to the presence of weak instruments, and that inference based on the DRML estimator can be misleading. 

Furthermore, the analysis for the Promo 2 experiments identified 4 out of 40 experiments where the score confidence set was unbounded. In all of these four cases, the DRML confidence interval included zero. However, there is one experiment where both confidence sets are bounded, but the two methods disagree in terms of significance: while the score confidence set does include the zero, the DRML confidence set does not.

Now, focusing on the set of experiments for which both confidence sets are bounded, we report the relationship between the diameter of the score confidence set, the diameter of the DRML confidence interval and the sample size. Instead of the full sample size, we use the sample size of the instrument treatment group, which is in all cases the smallest of the two groups. Moreover, in order to aid the visualization of the results, we take the logarithm of the instrument treatment group sample size, and apply a min-max normalization to it.
Figure \ref{fig:glovo_chop} below shows that, excepting a few outliers for Promo 1 experiments, the ratio of the diameters of the confidence sets obtained with the two methods converges to 1 as the sample size increases, in line with Theorem \ref{theo:main}.

\begin{figure}
    \centering
    \includegraphics[scale=0.7]{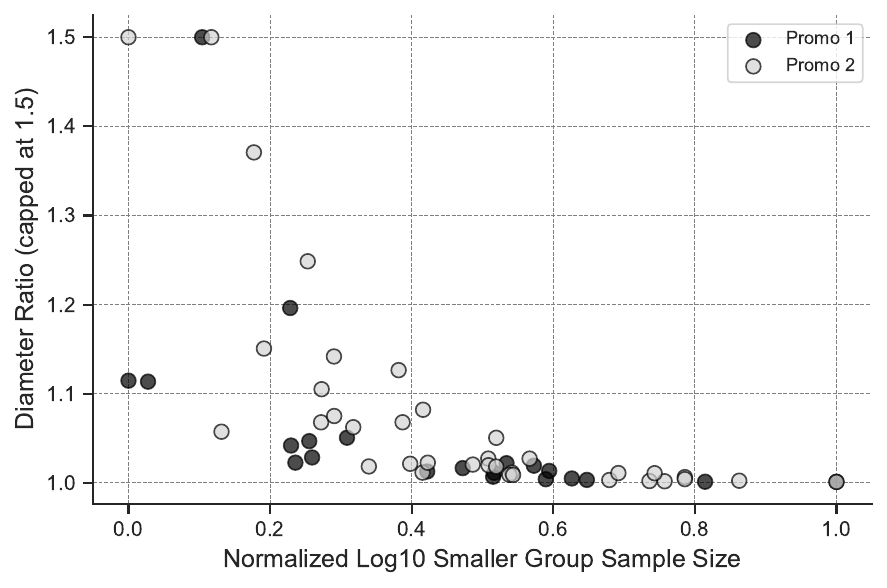}
    
    \caption{Ratio of the diameters of the score confidence set and the DRML Wald confidence interval for Promo 1 and Promo 2 experiments. The x-axis shows the min-max normalized logarithm of the size of the instrument treatment group.}
    \label{fig:glovo_chop}
\end{figure}
\FloatBarrier

\section*{Acknowledgements}
The authors thank Jamie Robins, Andrea Rotnitzky and Shiro Theuri for helpful discussions.

\section*{Appendix}
We include here the proofs of all the results stated in the main text.

\begin{proof}[Proof of Theorem \ref{theo:main}]
    
    Note that, under the assumptions of Theorem \ref{theo:main}, we have that
    the following hold under $P$:
    \begin{itemize}
        \item $\mathbb{P}_{n} \left( \psi_{a,\widehat{\eta}}\right)^{i} \convP E_{P} \left( \psi_{a,\eta_{P}}\right)^{i}$ for $i=1,2$,
        \item $\mathbb{P}_{n} \left( \psi_{b,\widehat{\eta}}\right)^{i} \convP E_{P} \left( \psi_{b,\eta_{P}}\right)^{i}$ for $i=1,2$,  
        \item $\mathbb{P}_{n} \left( \psi_{a,\widehat{\eta}}\psi_{b,\widehat{\eta}}\right) \convP E_{P} \left( \psi_{a,\eta_{P}}\psi_{b,\eta_{P}}\right)$.
    \end{itemize}

    We will show that
    $$
        P\left( a>0, \Delta > 0\right) \to 1.
    $$
    From the assumptions, we have that
    \begin{align*}
        \frac{a}{n}= &\left[ \mathbb{P}_{n} \psi_{a,\widehat{\eta}}\right]^{2} - (1/n)z_{1-\alpha/2}^{2} \mathbb{P}_{n} \left[ \psi_{a,\widehat{\eta}}\right]^{2} \convP E_{P}(\psi_{a,\eta_{P}})^{2}  > 0.
    \end{align*}
    Thus, $P\left( a>0\right) \to 1$. On the other hand,
    \begin{align}
        &\frac{\Delta}{n}=
        \nonumber
        \\
        & 
        4n \left[ \mathbb{P}_{n} \psi_{a,\widehat{\eta}}\right]^{2}\left[ \mathbb{P}_{n} \psi_{b,\widehat{\eta}}\right]^{2} - 4n \left[ \mathbb{P}_{n} \psi_{a,\widehat{\eta}}\right]^{2}\left[ \mathbb{P}_{n} \psi_{b,\widehat{\eta}}\right]^{2}+(4/n)z_{1-\alpha/2}^{4} \left[ \mathbb{P}_{n} \psi_{a,\widehat{\eta}}\psi_{b,\widehat{\eta}}\right]^{2} -
        \nonumber
        \\
        &(4/n)z_{1-\alpha/2}^{4} \left[ \mathbb{P}_{n} \psi_{a,\widehat{\eta}}\right]^{2}\left[ \mathbb{P}_{n} \psi_{b,\widehat{\eta}}\right]^{2} - 8 z_{1-\alpha/2}^{2} \mathbb{P}_{n} \psi_{a,\widehat{\eta}}\psi_{b,\widehat{\eta}} \mathbb{P}_{n} \psi_{a,\widehat{\eta}} \mathbb{P}_{n} \psi_{b,\widehat{\eta}} +
        \nonumber
        \\
        & 4 z_{1-\alpha_{2}}^{2}\left\lbrace \left[ \mathbb{P}_{n} \psi_{a,\widehat{\eta}}\right]^{2}\mathbb{P}_{n} \psi^{2}_{b,\widehat{\eta}} + \left[ \mathbb{P}_{n} \psi_{b,\widehat{\eta}}\right]^{2}\mathbb{P}_{n} \psi^{2}_{a,\widehat{\eta}} \right\rbrace \convP
        \nonumber
        \\
        & 4 z_{1-\alpha_{2}}^{2} \left\lbrace E^{2}_{P} \psi_{a,\eta_{P}} E_{P} \psi^{2}_{b,\eta_{P}} + E^{2}_{P} \psi_{b,\eta_{P}} E_{P} \psi^{2}_{a,\eta_{P}} - 2E_{P}\psi_{a,\eta_{P}}E_{P}\psi_{b,\eta_{P}}E_{P}\psi_{a,\eta_{P}} \psi_{b,\eta_{P}}\right\rbrace=
        \nonumber
        \\
        &4 z_{1-\alpha_{2}}^{2} E^{2}_{P} \psi_{a,\eta_{P}} E_{P}\left\lbrace \psi_{b,\eta_{P}}+\varphi(P) \psi_{a,\eta_{P}} \right\rbrace^{2} > 0,
        \label{eq:determinant_is_asymptotic_variance}
    \end{align}
    which implies that $P\left( \Delta > 0\right) \to 1$. Hence $P\left( a>0, \Delta > 0\right) \to 1$. Thus, $P\left( a>0, \Delta > 0\right) \to 1$, which by Proposition \ref{prop:five_forms} implies $P(C_{n}=[r_{1},r_{2}])\to 1$.

    To finish the proof, we need to show that $r_{1}=\widehat{\varphi}-z_{1-\alpha/2}(\widehat{\sigma}/\sqrt{n})+O_{P}(1/n)$ and $r_{2}=\widehat{\varphi}+z_{1-\alpha/2}(\widehat{\sigma}/\sqrt{n})+O_{P}(1/n)$. 
    Note that 
    \begin{align}
        &n\left( \frac{-b}{2a} - \widehat{\varphi}\right)=
        \\
        &n\left(\frac{\mathbb{P}_{n} \psi_{a,\widehat{\eta}}\mathbb{P}_{n} \psi_{b,\widehat{\eta}} - (1/n)z_{1-\alpha/2}^{2} \mathbb{P}_{n} \left[ \psi_{a,\widehat{\eta}}\psi_{b,\widehat{\eta}}\right]}{\left[ \mathbb{P}_{n} \psi_{a,\widehat{\eta}}\right]^{2} - (1/n)z_{1-\alpha/2}^{2} \mathbb{P}_{n} \left[ \psi_{a,\widehat{\eta}}\right]^{2}} - \frac{\mathbb{P}_{n}\psi_{b,\widehat{\eta}}}{\mathbb{P}_{n}\psi_{a,\widehat{\eta}}} \right)=
        \nonumber
        \\
        &\frac{n\mathbb{P}^{2}_{n} \psi_{a,\widehat{\eta}}\mathbb{P}_{n} \psi_{b,\widehat{\eta}} - z_{1-\alpha/2}^{2} \mathbb{P}_{n} \left[ \psi_{a,\widehat{\eta}}\psi_{b,\widehat{\eta}}\right]\mathbb{P}_{n}\psi_{a,\widehat{\eta}}-n\mathbb{P}^{2}_{n} \psi_{a,\widehat{\eta}}\mathbb{P}_{n} \psi_{b,\widehat{\eta}}+z_{1-\alpha/2}^{2} \mathbb{P}_{n} \left[ \psi_{a,\widehat{\eta}}\right]^{2}\mathbb{P}_{n}\psi_{b,\widehat{\eta}}}{\left\lbrace\left[ \mathbb{P}_{n} \psi_{a,\widehat{\eta}}\right]^{2} - (1/n)z_{1-\alpha/2}^{2} \mathbb{P}_{n} \left[ \psi_{a,\widehat{\eta}}\right]^{2}\right\rbrace \mathbb{P}_{n}\psi_{a,\widehat{\eta}}}=
        \nonumber
        \\
        &\frac{-z_{1-\alpha/2}^{2} \mathbb{P}_{n} \left[ \psi_{a,\widehat{\eta}}\psi_{b,\widehat{\eta}}\right]\mathbb{P}_{n}\psi_{a,\widehat{\eta}}+z_{1-\alpha/2}^{2} \mathbb{P}_{n} \left[ \psi_{a,\widehat{\eta}}\right]^{2}\mathbb{P}_{n}\psi_{b,\widehat{\eta}}}{\left\lbrace\left[ \mathbb{P}_{n} \psi_{a,\widehat{\eta}}\right]^{2} - (1/n)z_{1-\alpha/2}^{2} \mathbb{P}_{n} \left[ \psi_{a,\widehat{\eta}}\right]^{2}\right\rbrace \mathbb{P}_{n}\psi_{a,\widehat{\eta}}}\convP
        \nonumber
        \\
        &\frac{-z_{1-\alpha/2}^{2} E_{P} \left[ \psi_{a,\eta_{P}}\psi_{b,\eta_{P}}\right]E_{P}\psi_{a,\eta_{P}}+z_{1-\alpha/2}^{2} E_{P} \left[ \psi_{a,\eta_{P}}\right]^{2}E_{P}\psi_{b,\eta_{P}}}{\left\lbrace E_{P}\left[ \psi_{a,\eta_{P}}\right]\right\rbrace^{3}}.
        \nonumber
    \end{align}
    Thus, 
    \begin{equation}
        n\left( \frac{-b}{2a} - \widehat{\varphi}\right)=O_{P}(1).
        \label{eq:first_term_Op}
    \end{equation}

    Let
    $$
    \gamma_{n}=\left[ \mathbb{P}_{n} \psi_{a,\widehat{\eta}}\right]^{2}\mathbb{P}_{n} \psi^{2}_{b,\widehat{\eta}} + \left[ \mathbb{P}_{n} \psi_{b,\widehat{\eta}}\right]^{2}\mathbb{P}_{n} \psi^{2}_{a,\widehat{\eta}} - 2 \mathbb{P}_{n} \psi_{a,\widehat{\eta}}\psi_{b,\widehat{\eta}} \mathbb{P}_{n} \psi_{a,\widehat{\eta}} \mathbb{P}_{n} \psi_{b,\widehat{\eta}}
    $$
    and
    \begin{equation}
    \sigma^{2}_{P}=\frac{E_{P}\left\lbrace \psi_{b,\eta_{P}}-\varphi(P) \psi_{a,\eta_{P}} \right\rbrace^{2}}{E^{2}_{P}\psi_{a,\eta_{P}}}=var_{P}\left( \varphi^{1}_{P} \right).
    \label{eq:def_sigmap}
    \end{equation}
    Note that  $\widehat{\sigma}^{2}\convP \sigma^{2}_{P}$ and  $\gamma_{n}=\widehat{\sigma}^{2}\mathbb{P}^{4}_{n}\psi_{a,\widehat{\eta}}$.
    Then 
    \begin{align*}
        &n\left( \frac{n}{z^{2}_{1-\alpha/2} \widehat{\sigma}^{2}}\frac{\Delta}{4 a^{2}} -1 \right)= 
        \\
        &n\left( \frac{n}{z^{2}_{1-\alpha/2} \widehat{\sigma}^{2}}\frac{4 z_{1-\alpha/2}^{4} \left\lbrace\left[ \mathbb{P}_{n} \psi_{a,\widehat{\eta}}\psi_{b,\widehat{\eta}}\right]^{2} - \left[ \mathbb{P}_{n} \psi_{a,\widehat{\eta}}\right]^{2}\left[ \mathbb{P}_{n} \psi_{b,\widehat{\eta}}\right]^{2}\right\rbrace + 4 n z^{2}_{1-\alpha/2} \gamma_{n}}{4n^{2}  \left[ \mathbb{P}_{n} \psi_{a,\widehat{\eta}}\right]^{4} + 4 z_{1-\alpha/2}^{4} \mathbb{P}_{n} \left[ \psi_{a,\widehat{\eta}}\right]^{4} - 8nz_{1-\alpha/2}^{2}\left[ \mathbb{P}_{n} \psi_{a,\widehat{\eta}}\right]^{2}\mathbb{P}_{n} \left[ \psi_{a,\widehat{\eta}}\right]^{2} } - 1\right)=
        &
        \\
        &n\left( \frac{\gamma_{n}+ (z^{2}_{1-\alpha/2}/n)\left\lbrace\left[ \mathbb{P}_{n} \psi_{a,\widehat{\eta}}\psi_{b,\widehat{\eta}}\right]^{2} - \left[ \mathbb{P}_{n} \psi_{a,\widehat{\eta}}\right]^{2}\left[ \mathbb{P}_{n} \psi_{b,\widehat{\eta}}\right]^{2}\right\rbrace}{\widehat{\sigma}^{2}\mathbb{P}_{n} \left[ \psi_{a,\widehat{\eta}}\right]^{4}  +\widehat{\sigma}^{2}(z_{1-\alpha/2}^{4}/n^{2}) \mathbb{P}_{n} \left[ \psi_{a,\widehat{\eta}}\right]^{4}-(2/n)\widehat{\sigma}^{2}z^{2}_{1-\alpha/2}\left[ \mathbb{P}_{n} \psi_{a,\widehat{\eta}}\right]^{2}\mathbb{P}_{n} \left[ \psi_{a,\widehat{\eta}}\right]^{2}} - 1\right)=
        \\
        &\frac{n\gamma_{n}+ (z^{2}_{1-\alpha/2})\left\lbrace\left[ \mathbb{P}_{n} \psi_{a,\widehat{\eta}}\psi_{b,\widehat{\eta}}\right]^{2} - \left[ \mathbb{P}_{n} \psi_{a,\widehat{\eta}}\right]^{2}\left[ \mathbb{P}_{n} \psi_{b,\widehat{\eta}}\right]^{2}\right\rbrace - n\gamma_{n} -\widehat{\sigma}^{2}(z_{1-\alpha/2}^{4}/n) \mathbb{P}_{n} \left[ \psi_{a,\widehat{\eta}}\right]^{4}+2\widehat{\sigma}^{2}z^{2}_{1-\alpha/2}\left[ \mathbb{P}_{n} \psi_{a,\widehat{\eta}}\right]^{2}\mathbb{P}_{n} \left[ \psi_{a,\widehat{\eta}}\right]^{2}}{\gamma_{n} +\widehat{\sigma}^{2}(z_{1-\alpha/2}^{4}/n^{2}) \mathbb{P}_{n} \left[ \psi_{a,\widehat{\eta}}\right]^{4}-(2/n)\widehat{\sigma}^{2}z^{2}_{1-\alpha/2}\left[ \mathbb{P}_{n} \psi_{a,\widehat{\eta}}\right]^{2}\mathbb{P}_{n} \left[ \psi_{a,\widehat{\eta}}\right]^{2}}
        \\
        &\frac{ (z^{2}_{1-\alpha/2})\left\lbrace\left[ \mathbb{P}_{n} \psi_{a,\widehat{\eta}}\psi_{b,\widehat{\eta}}\right]^{2} - \left[ \mathbb{P}_{n} \psi_{a,\widehat{\eta}}\right]^{2}\left[ \mathbb{P}_{n} \psi_{a,\widehat{\eta}}\right]^{2}\right\rbrace  -\widehat{\sigma}^{2}(z_{1-\alpha/2}^{4}/n) \mathbb{P}_{n} \left[ \psi_{a,\widehat{\eta}}\right]^{4}+2\widehat{\sigma}^{2}z^{2}_{1-\alpha/2}\left[ \mathbb{P}_{n} \psi_{a,\widehat{\eta}}\right]^{2}\mathbb{P}_{n} \left[ \psi_{a,\widehat{\eta}}\right]^{2}}{\gamma_{n} +\widehat{\sigma}^{2}(z_{1-\alpha/2}^{4}/n^{2}) \mathbb{P}_{n} \left[ \psi_{a,\widehat{\eta}}\right]^{4}-(2/n)\widehat{\sigma}^{2}z^{2}_{1-\alpha/2}\left[ \mathbb{P}_{n} \psi_{a,\widehat{\eta}}\right]^{2}\mathbb{P}_{n} \left[ \psi_{a,\widehat{\eta}}\right]^{2}} \convP
        \\
        &\frac{ (z^{2}_{1-\alpha/2})\left\lbrace\left[ E_{P} \psi_{a,\eta_{P}}\psi_{b,\eta_{P}}\right]^{2} - \left[ E_{P} \psi_{a,\eta_{P}}\right]^{2}\left[ E_{P} \psi_{a,\eta_{P}}\right]^{2}\right\rbrace  +2\sigma_{P}^{2}z^{2}_{1-\alpha/2}\left[ E_{P} \psi_{a,\eta_{P}}\right]^{2}E_{P} \left[ \psi_{a,\eta_{P}}\right]^{2}}{\sigma^{2}_{P}E^{4}_{P}\psi_{a,\eta_{P}}}.
    \end{align*}
    Thus 
\begin{equation}
\frac{\sqrt{\Delta}}{2a}=z_{1-\alpha/2} \widehat{\sigma}/\sqrt{n}+O_{P}(1).
\label{eq:second_term_Op}
\end{equation}
The result now follows from \eqref{eq:first_term_Op} and \eqref{eq:second_term_Op}.
\end{proof}

\begin{proof}[Proof of Corollary \ref{coro:main}]
 Recall the definition of $\sigma^{2}_{P}$ in \eqref{eq:def_sigmap}. By Theorem 25.20 of \cite{van2000asymptotic}, we have that $\sigma_{P}<s_{P}$. 
 By Theorem \ref{theo:main}, we have that, with probability tending to one,
 \begin{align*}
    \frac{\sqrt{n}}{(2z_{1-\alpha/2})}diam(C_{n}) = \widehat{\sigma} + O_{P}(1/\sqrt{n}).
 \end{align*}
 On the other hand, we have that
 \begin{align*}
    \frac{\sqrt{n}}{(2z_{1-\alpha/2})}diam(I_{n}) = \widetilde{s}.
 \end{align*}

 It follows that, with probability tending to one, $diam(C_{n}) < diam(I_{n})$ if and only if
 $ \widehat{\sigma} + O_{P}(1/\sqrt{n}) <  \widetilde{s}$. Now, fix $\varepsilon>0$ such that $(1-\varepsilon)s_{P}>\sigma_{P}(1+\varepsilon)+\varepsilon$. Since by assumption $\widehat{\sigma}^{2}\convP \sigma^{2}_{P}$ and $\widetilde{s}^{2}\convP s^{2}_{P}$, we have that, with probability tending to one,
 $$
 \widehat{\sigma} + O_{P}(1/\sqrt{n})\leq \sigma_{P}(1+\varepsilon)+\varepsilon, \quad \text{and} \quad \widetilde{s}\geq (1-\varepsilon)s_{P}.
 $$
 Thus, with probability tending to one,
 $$
 \widehat{\sigma} + O_{P}(1/\sqrt{n})\leq \sigma_{P}(1+\varepsilon)+\varepsilon < (1-\varepsilon)s_{P} \leq \widetilde{s},
 $$
 and hence $diam(C_{n})<diam(I_{n})$. This finishes the proof of the corollary.


\end{proof}

To prove Theorem \ref{theo:weak_IV} we will need the following
lemmas.

\begin{lemma}\label{lemma:lindberg}
    Assume Condition \ref{cond:weak_IV} holds. Then $ \sqrt{n}\left( \mathbb{P}_{n}\psi_{a,\eta_{P_{n}}} - E_{P_{n}}\psi_{a,\eta_{P_{n}}}, \mathbb{P}_{n}\psi_{b,\eta_{P_{n}}} - E_{P_{n}}\psi_{b,\eta_{P_{n}}} \right)$ converges in distribution to a bivariate normal distribution with mean zero and covariance matrix $\Sigma_{ab}$.
\end{lemma}
\begin{proof}
    Since by assumption $\Sigma_{P_{n},ab}$ converges to $\Sigma_{ab}$, it suffices to show that  
    $$
    \Sigma^{-1/2}_{P_{n},ab} \sqrt{n}\left( \mathbb{P}_{n}\psi_{a,\eta_{P_{n}}} - E_{P_{n}}\psi_{a,\eta_{P_{n}}}, \mathbb{P}_{n}\psi_{b,\eta_{P_{n}}} - E_{P_{n}}\psi_{b,\eta_{P_{n}}} \right)
    $$ 
    converges in distribution to a bivariate standard normal distribution.
    To do so, we will
    verify the conditions of the Lindberg-Feller theorem. Let $\Vert \cdot\Vert$ be the Euclidean norm, and let $V=(\psi_{a,\eta_{P_{n}}},\psi_{b,\eta_{P_{n}}})$.
    Let $\lambda_{min}(\Sigma_{P_{n},ab})$ denote the smallest eigenvalue of $\Sigma_{P_{n},ab}$. We need to show that for any $\varepsilon>0$
    \begin{align*}
    &\lambda^{-2}_{min}(\Sigma_{P_{n},ab})  E_{P_{n}}\left[  \left\Vert V -  E_{P_{n}}V\right\Vert ^{2} I\left\lbrace \left\Vert V -  E_{P_{n}}V\right\Vert^{2} \geq \varepsilon n \lambda^{2}_{min}(\Sigma_{P_{n},ab}) \right\rbrace \right]\to 0.
    \end{align*}
    
    Since $\Sigma_{P_{n},ab}$ converges to $\Sigma_{ab}$, which is invertible, we have that there exists a constant $c_{\lambda}>0$ such that for all sufficiently large $n$ it holds that
     $\lambda_{min}(\Sigma_{P_{n},ab})\geq c_{\lambda}$. Then, using the other assumptions in Condition \ref{cond:weak_IV} we have that
    \begin{align*}
        &\lambda^{-2}_{min}(\Sigma_{P_{n},ab})  E_{P_{n}}\left[  \left\Vert V -  E_{P_{n}}V\right\Vert ^{2} I\left\lbrace \left\Vert V -  E_{P_{n}}V\right\Vert^{2} \geq \varepsilon n \lambda^{2}_{min}(\Sigma_{P_{n},ab}) \right\rbrace \right]
        \leq
        \\
        &c^{-2}_{\lambda}  E_{P_{n}}\left[  \left\Vert V -  E_{P_{n}}V\right\Vert ^{2} I\left\lbrace \left\Vert V -  E_{P_{n}}V\right\Vert^{2} \geq \varepsilon n c_{\lambda}^{2} \right\rbrace \right]
        \leq
        \\
        &c^{-2}_{\lambda}  E_{P_{n}}\left[  \left( \sqrt{2}c+ \left(\sqrt{c^{2}_{a}+c_{b}^{2}}\right)/\sqrt{n}\right)^{2} I\left\lbrace\left( \sqrt{2}c+ \left(\sqrt{c^{2}_{a}+c_{b}^{2}}\right)/\sqrt{n}\right)^{2} \geq \varepsilon n c_{\lambda}^{2} \right\rbrace \right]=0
        \end{align*}
    for sufficiently large $n$, which proves the Lemma.
\end{proof}
\begin{lemma}\label{lemma:ratio_conv}
    Assume Condition \ref{cond:weak_IV} holds. Then the distribution of
    $$
    \frac{\mathbb{P}_{n}\psi_{b,\eta_{P_{n}}}}{\mathbb{P}_{n}\psi_{a,\eta_{P_{n}}}} - \frac{E_{P_{n}}\psi_{b,\eta_{P_{n}}}}{E_{P_{n}}\psi_{a,\eta_{P_{n}}}}
    $$
    converges to the distribution of 
    $$
    \frac{N_{b}+c_{b}}{c_{a}-N_{a}} - \frac{c_{b}}{c_{a}},
    $$
    where $(N_{a},N_{b})$ has distribution $N(0, \Sigma_{ab})$.
\end{lemma}
\begin{proof}
    \begin{align*}
        \frac{\mathbb{P}_{n}\psi_{b,\eta_{P_{n}}}}{\mathbb{P}_{n}\psi_{a,\eta_{P_{n}}}} - \frac{E_{P_{n}}(\psi_{b,\eta_{P_{n}}})}{E_{P_{n}}(\psi_{a,\eta_{P_{n}}})}&=\frac{\mathbb{P}_{n}\psi_{b,\eta_{P_{n}}}}{\mathbb{P}_{n}\psi_{a,\eta_{P_{n}}}} - \frac{\mathbb{P}_{n}\psi_{b,\eta_{P_{n}}}}{E_{P_{n}}(\psi_{a,\eta_{P_{n}}})} + \frac{\mathbb{P}_{n}\psi_{b,\eta_{P_{n}}}}{E_{P_{n}}(\psi_{a,\eta_{P_{n}}})}-\frac{E_{P_{n}}(\psi_{b,\eta_{P_{n}}})}{E_{P_{n}}(\psi_{a,\eta_{P_{n}}})}
        \\
        &=
        \frac{\mathbb{P}_{n}\psi_{b,\eta_{P_{n}}}}{\mathbb{P}_{n}\psi_{a,\eta_{P_{n}}}E_{P_{n}}(\psi_{a,\eta_{P_{n}}})} \left\lbrace E_{P_{n}}(\psi_{a,\eta_{P_{n}}}) -\mathbb{P}_{n}\psi_{a,\eta_{P_{n}}} \right\rbrace + \frac{\sqrt{n}}{c_{a}}\left\lbrace \mathbb{P}_{n}\psi_{b,\eta_{P_{n}}} - E_{P_{n}}(\psi_{b,\eta_{P_{n}}}) \right\rbrace
        \\
        &=
        \frac{\mathbb{P}_{n}\psi_{b,\eta_{P_{n}}}}{\mathbb{P}_{n}\psi_{a,\eta_{P_{n}}}} \frac{\sqrt{n}\left\lbrace E_{P_{n}}(\psi_{a,\eta_{P_{n}}}) -\mathbb{P}_{n}\psi_{a,\eta_{P_{n}}} \right\rbrace}{c_{a}} + \frac{\sqrt{n}}{c_{a}}\left\lbrace \mathbb{P}_{n}\psi_{b,\eta_{P_{n}}} - E_{P_{n}}(\psi_{b,\eta_{P_{n}}}) \right\rbrace.
    \end{align*}
    Thus 
    \begin{align*}
        \frac{\mathbb{P}_{n}\psi_{b,\eta_{P_{n}}}}{\mathbb{P}_{n}\psi_{a,\eta_{P_{n}}}} \left[ 1 - \frac{\sqrt{n}\left\lbrace E_{P_{n}}(\psi_{a,\eta_{P_{n}}}) -\mathbb{P}_{n}\psi_{a,\eta_{P_{n}}} \right\rbrace}{c_{a}} \right] = \frac{\sqrt{n}}{c_{a}}\left\lbrace \mathbb{P}_{n}\psi_{b,\eta_{P_{n}}} - E_{P_{n}}(\psi_{b,\eta_{P_{n}}}) \right\rbrace + \frac{E_{P_{n}}(\psi_{b,\eta_{P_{n}}})}{E_{P_{n}}(\psi_{a,\eta_{P_{n}}})}
    \end{align*}
    which implies that
    \begin{align*}
        \frac{\mathbb{P}_{n}\psi_{b,\eta_{P_{n}}}}{\mathbb{P}_{n}\psi_{a,\eta_{P_{n}}}} - \frac{E_{P_{n}}(\psi_{b,\eta_{P_{n}}})}{E_{P_{n}}(\psi_{a,\eta_{P_{n}}})}&= \frac{\frac{\sqrt{n}}{c_{a}}\left\lbrace \mathbb{P}_{n}\psi_{b,\eta_{P_{n}}} - E_{P_{n}}(\psi_{b,\eta_{P_{n}}}) \right\rbrace + \frac{E_{P_{n}}(\psi_{b,\eta_{P_{n}}})}{E_{P_{n}}(\psi_{a,\eta_{P_{n}}})}}{\left[ 1 - \frac{\sqrt{n}\left\lbrace E_{P_{n}}(\psi_{a,\eta_{P_{n}}}) -\mathbb{P}_{n}\psi_{a,\eta_{P_{n}}} \right\rbrace}{c_{a}} \right]} - \frac{E_{P_{n}}(\psi_{b,\eta_{P_{n}}})}{E_{P_{n}}(\psi_{a,\eta_{P_{n}}})}
        \\
        &= \frac{\frac{\sqrt{n}}{c_{a}}\left\lbrace \mathbb{P}_{n}\psi_{b,\eta_{P_{n}}} - E_{P_{n}}(\psi_{b,\eta_{P_{n}}}) \right\rbrace + c_{b}/c_{a}}{\left[ 1 - \frac{\sqrt{n}\left\lbrace E_{P_{n}}(\psi_{a,\eta_{P_{n}}}) -\mathbb{P}_{n}\psi_{a,\eta_{P_{n}}} \right\rbrace}{c_{a}} \right]} - \frac{c_{b}}{c_{a}}
        \\
        &= \frac{\sqrt{n}\left\lbrace \mathbb{P}_{n}\psi_{b,\eta_{P_{n}}} - E_{P_{n}}(\psi_{b,\eta_{P_{n}}}) \right\rbrace + c_{b}}{\left[ c_{a} - \sqrt{n}\left\lbrace E_{P_{n}}(\psi_{a,\eta_{P_{n}}}) -\mathbb{P}_{n}\psi_{a,\eta_{P_{n}}} \right\rbrace \right]} - \frac{c_{b}}{c_{a}}.
    \end{align*}
    By Lemma \ref{lemma:lindberg}, we have that 
    $\sqrt{n}\left( \mathbb{P}_{n}\psi_{a,\eta_{P_{n}}} - E_{P_{n}}(\psi_{a,\eta_{P_{n}}}), \mathbb{P}_{n}\psi_{b,\eta_{P_{n}}} - E_{P_{n}}(\psi_{b,\eta_{P_{n}}}) \right)$ converges in distribution to a $N(0, \Sigma_{ab})$ distribution. Thus, we conclude that the distribution of
    $$
    \frac{\mathbb{P}_{n}\psi_{b,\eta_{P_{n}}}}{\mathbb{P}_{n}\psi_{a,\eta_{P_{n}}}} - \frac{E_{P_{n}}(\psi_{b,\eta_{P_{n}}})}{E_{P_{n}}(\psi_{a,\eta_{P_{n}}})}
    $$
    converges to the distribution of 
    $$
    \frac{N_{b}+c_{b}}{c_{a}-N_{a}} - \frac{c_{b}}{c_{a}},
    $$
    where $(N_{a},N_{b})$ has distribution $N(0, \Sigma_{ab})$, which is what we wanted to show.
\end{proof}

We are now ready to prove Theorem \ref{theo:weak_IV}.
\begin{proof}[Proof of Theorem \ref{theo:weak_IV}]
    \begin{align*}
    \widehat{\varphi} - \varphi(P_{n})&=\frac{\mathbb{P}_{n}\psi_{b,\widehat{\eta}}}{\mathbb{P}_{n}\psi_{a,\widehat{\eta}}} - \frac{E_{P_{n}} \psi_{b,\eta_{P_{n}}}}{E_{P_{n}} \psi_{a,\eta_{P_{n}}}}
    \\
    &
    =\frac{\mathbb{P}_{n}\psi_{b,\widehat{\eta}}}{\mathbb{P}_{n}\psi_{a,\widehat{\eta}}} - \frac{\mathbb{P}_{n}\psi_{b,\widehat{\eta}}}{\mathbb{P}_{n}\psi_{a,\eta_{P_{n}}}}
    + \frac{\mathbb{P}_{n}\psi_{b,\widehat{\eta}}}{\mathbb{P}_{n}\psi_{a,\eta_{P_{n}}}} - \frac{\mathbb{P}_{n}\psi_{b,\eta_{P_{n}}}}{\mathbb{P}_{n}\psi_{a,\eta_{P_{n}}}}
    +\frac{\mathbb{P}_{n}\psi_{b,\eta_{P_{n}}}}{\mathbb{P}_{n}\psi_{a,\eta_{P_{n}}}}-
    \frac{E_{P_{n}} \psi_{b,\eta_{P_{n}}}}{E_{P_{n}} \psi_{a,\eta_{P_{n}}}}.
    \end{align*}
    Define 
    $$
    D_{n}:= \frac{\mathbb{P}_{n}\psi_{b,\widehat{\eta}}}{\mathbb{P}_{n}\psi_{a,\widehat{\eta}}} - \frac{\mathbb{P}_{n}\psi_{b,\widehat{\eta}}}{\mathbb{P}_{n}\psi_{a,\eta_{P_{n}}}}, \quad B_{n}:= \frac{\mathbb{P}_{n}\psi_{b,\widehat{\eta}}}{\mathbb{P}_{n}\psi_{a,\eta_{P_{n}}}} - \frac{\mathbb{P}_{n}\psi_{b,\eta_{P_{n}}}}{\mathbb{P}_{n}\psi_{a,\eta_{P_{n}}}}, \quad V_{n}:=\frac{\mathbb{P}_{n}\psi_{b,\eta_{P_{n}}}}{\mathbb{P}_{n}\psi_{a,\eta_{P_{n}}}}-
    \frac{E_{P_{n}} \psi_{b,\eta_{P_{n}}}}{E_{P_{n}} \psi_{a,\eta_{P_{n}}}}.
    $$
    Then
    \begin{align*}
        D_{n}=\frac{\mathbb{P}_{n}\psi_{b,\widehat{\eta}}}{\mathbb{P}_{n}\psi_{a,\widehat{\eta}}} 
        \frac{\left\lbrace \mathbb{P}_{n}\psi_{a,\eta_{P_{n}}} -\mathbb{P}_{n}\psi_{a,\widehat{\eta}}\right\rbrace}{\mathbb{P}_{n}\psi_{a,\eta_{P_{n}}}}&=
        \frac{\mathbb{P}_{n}\psi_{b,\widehat{\eta}}}{\mathbb{P}_{n}\psi_{a,\widehat{\eta}}} 
        \frac{\left\lbrace \mathbb{P}_{n}\psi_{a,\eta_{P_{n}}} -\mathbb{P}_{n}\psi_{a,\widehat{\eta}}\right\rbrace}{\mathbb{P}_{n}\psi_{a,\eta_{P_{n}}}-E_{P_{n}}\psi_{a,\eta_{P_{n}}} + E_{P_{n}}\psi_{a,\eta_{P_{n}}}}
        \\
        &
        =\frac{\mathbb{P}_{n}\psi_{b,\widehat{\eta}}}{\mathbb{P}_{n}\psi_{a,\widehat{\eta}}} 
        \frac{\sqrt{n}\left\lbrace \mathbb{P}_{n}\psi_{a,\eta_{P_{n}}} -\mathbb{P}_{n}\psi_{a,\widehat{\eta}}\right\rbrace}{\sqrt{n}\left\lbrace\mathbb{P}_{n}\psi_{a,\eta_{P_{n}}}-E_{P_{n}}\psi_{a,\eta_{P_{n}}}\right\rbrace + c_{a}}.
    \end{align*}
    
    Thus 
    \begin{align*}
        \widehat{\varphi} - \varphi(P_{n})&= \frac{\mathbb{P}_{n}\psi_{b,\widehat{\eta}}}{\mathbb{P}_{n}\psi_{a,\widehat{\eta}}} 
        \frac{\sqrt{n}\left\lbrace \mathbb{P}_{n}\psi_{a,\eta_{P_{n}}} -\mathbb{P}_{n}\psi_{a,\widehat{\eta}}\right\rbrace}{\sqrt{n}\left\lbrace\mathbb{P}_{n}\psi_{a,\eta_{P_{n}}}-E_{P_{n}}\psi_{a,\eta_{P_{n}}}\right\rbrace + c_{a}} + B_{n}+ V_{n}.
    \end{align*}
    It follows that 
    \begin{align*}
        \frac{\mathbb{P}_{n}\psi_{b,\widehat{\eta}}}{\mathbb{P}_{n}\psi_{a,\widehat{\eta}}} \left[ 1 - \frac{\sqrt{n}\left\lbrace \mathbb{P}_{n}\psi_{a,\eta_{P_{n}}} -\mathbb{P}_{n}\psi_{a,\widehat{\eta}}\right\rbrace}{\sqrt{n}\left\lbrace\mathbb{P}_{n}\psi_{a,\eta_{P_{n}}}-E_{P_{n}}\psi_{a,\eta_{P_{n}}}\right\rbrace + c_{a}} \right]&= 
         B_{n}+ V_{n}  + \varphi(P_{n})
    \end{align*}
    and hence, using the fact that $\varphi(P_{n})=c_{b}/c_{a}$
    \begin{align*}
        \widehat{\varphi} - \varphi(P_{n})&= 
         \frac{B_{n}+ V_{n}  +c_{b}/c_{a}}{1 - \frac{\sqrt{n}\left\lbrace \mathbb{P}_{n}\psi_{a,\eta_{P_{n}}} -\mathbb{P}_{n}\psi_{a,\widehat{\eta}}\right\rbrace}{\sqrt{n}\left\lbrace\mathbb{P}_{n}\psi_{a,\eta_{P_{n}}}-E_{P_{n}}\psi_{a,\eta_{P_{n}}}\right\rbrace + c_{a}}} - \frac{c_{b}}{c_{a}}
    \end{align*}
    
    Now
    \begin{align}
        B_{n}= \frac{\mathbb{P}_{n}\psi_{b,\widehat{\eta}} - \mathbb{P}_{n}\psi_{b,\eta_{P_{n}}}}{\mathbb{P}_{n}\psi_{a,\eta_{P_{n}}}}=\frac{\mathbb{P}_{n}\psi_{b,\widehat{\eta}} - \mathbb{P}_{n}\psi_{b,\eta_{P_{n}}}}{\mathbb{P}_{n}\psi_{a,\eta_{P_{n}}}-E_{P_{n}}\psi_{a,\eta_{P_{n}}} + E_{P_{n}}\psi_{a,\eta_{P_{n}}}}=\frac{\sqrt{n}\left\lbrace \mathbb{P}_{n}\psi_{b,\widehat{\eta}} - \mathbb{P}_{n}\psi_{b,\eta_{P_{n}}}\right\rbrace}{\sqrt{n}\left\lbrace\mathbb{P}_{n}\psi_{a,\eta_{P_{n}}}-E_{P_{n}}\psi_{a,\eta_{P_{n}}}\right\rbrace + c_{a}}.
        \label{eq:ex_Bn}
    \end{align}
    By Condition \ref{cond:nuisances}, the numerator on the right-hand side of \eqref{eq:ex_Bn} converges to zero in probability under $P_{n}$. By Lemma \ref{lemma:lindberg} the denominator on the right hand side of the last equation in \eqref{eq:ex_Bn} converges to normal random variable with mean $c_{a}$. Thus, $B_{n}$ is $o_{P_{n}}(1)$.
    Using similar arguments, we can show that
    $$
    1 - \frac{\sqrt{n}\left\lbrace \mathbb{P}_{n}\psi_{a,\eta_{P_{n}}} -\mathbb{P}_{n}\psi_{a,\widehat{\eta}}\right\rbrace}{\sqrt{n}\left\lbrace\mathbb{P}_{n}\psi_{a,\eta_{P_{n}}}-E_{P_{n}}\psi_{a,\eta_{P_{n}}}\right\rbrace + c_{a}}
    $$
    converges in probability to $1$ under $P_{n}$.
    
    It follows that 
    $\widehat{\varphi} - \varphi(P_{n})$ has the same limit in distribution as $V_{n}$. By Lemma \ref{lemma:ratio_conv}, this is the distribution of 
    $$
    \frac{N_{b}+c_{b}}{c_{a}-N_{a}} - \frac{c_{b}}{c_{a}}=\frac{c_{a}(N_{b}+c_{b})-c_{b}(c_{a}-N_{a})}{c_{a}(c_{a}-N_{a})}=\frac{c_{a}N_{b}+c_{b}N_{a}}{c_{a}^{2}-c_{a}N_{a}},
    $$
    where $(N_{a},N_{b})$ has distribution $N(0, \Sigma_{ab})$. This finishes the proof of the theorem.
    \end{proof}

\bibliographystyle{apalike}
\bibliography{testability}

\end{document}